\documentclass{amsart}
\usepackage{graphicx}

\newtheorem{theorem}{Theorem}[section]
\newtheorem*{theorem A}{Theorem A}
\newtheorem*{theorem B}{N\"olker's Theorem}
\newtheorem{lemma}{Lemma}[section]

\newcommand{\N}{{\mathbb N}}
\newcommand{\Z}{{\mathbb Z}}

\theoremstyle{remark}
\newtheorem{remark}{Remark}[section]
\theoremstyle{remark}

\theoremstyle{definition}

\newtheorem{definition}{Definition}[section]

\numberwithin{equation}{section}
\def\({\left ( }
\def\){\right )}
\def\<{\left < }
\def\>{\right >}

\begin{document}

\vspace{2cm}

\title{Identities involving Narayana numbers}

\author[G. Cerda-Morales]{Gamaliel Cerda-Morales}
\address{Instituto de Matem\'aticas, Pontificia Universidad Cat\'olica de Valpara\'iso, Blanco Viel 596, Valpara\'iso, Chile.}
\email{gamaliel.cerda.m@mail.pucv.cl}






\begin{abstract}
Narayana's cows problem is a problem similar to the Fibonacci's rabbit problem. We define the numbers which are the solutions of this problem as Narayana's cows numbers. Narayana's cows sequence satisfies the third order recurrence relation $N_{r}=N_{r-1}+N_{r-3}$ ($r\geq3$) with initial condition $N_{0} =0$, $N_{1} = N_{2}= 1$. In this paper, the $am+b$ subscripted Narayana numbers will be expressed by three $a$ step apart Narayana numbers for any $1\leq b\leq a$ ($a\in \Z$). Furthermore, we study the sum $S_{N,r}^{(4,b)}=\sum_{k=0}^{r}N_{4k+b}$ of $4$ step apart Narayana numbers for any $1\leq b\leq 4$.
\end{abstract}
\maketitle

\section{Introduction}
The Narayana sequence was introduced by the Indian mathematician Narayana in the 14th century, while studying the following problem of a herd of cows and calves: A cow produces one calf every year. Beginning in its fourth year, each calf produces one calf at the beginning of each year. How many calves are there altogether after 20 years? (cf. \cite{Al-Jo}).

This problem can be solved in the same way that Fibonacci solved its problem about rabbits (cf. \cite{Ko}). If $r$ is the year, then the Narayana problem can be modelled by the recurrence
\begin{equation}
N_{r+1}=N_{r}+N_{r-2},\ N_{0}=0,\ N_{1}=N_{2}=1\ \ (r\geq2).
\end{equation}
The first few terms are $\{N_{r}\}_{r\geq0}=\{0,1,1,1,2,3,4,6,9,13,...\}$. This sequence is called Narayana sequence. The sequences $\{N_{r}\}_{r\geq0}$ can be defined for negative values of $r$ by using the definition $N_{-(s+1)}=-N_{-(s-1)}+N_{-(s-2)}$ ($s\geq2$) with initial conditions $N_{0}=N_{-1}=0$ and $N_{-2}=1$.

A number of properties of the Narayana sequence were studied in \cite{Ra-Si} using matrix methods and their generalization called $k$-Narayana sequence was studied. In \cite{Bi}, Bilgici defined a new recurrence which is called generalized order-$k$ Narayana's cows sequence and by using this generalization and some matrix properties, he gave some identities related to the Narayana's cows numbers. In this work we shall generalize the identities about $ar$ subscripted Narayana numbers $N_{ar}$ to any $N_{ar+b}$ ($1\leq b<a$). One of our main theorem is to express $N_{ar+b}$ by $N_{2a+b}$, $N_{a+b}$ and $N_{b}$, which are $a$ step apart terms.

\section{Narayana table}
For $a\in \N$, when we say $a$ columns Narayana table we mean a rectangle shape having $a$ columns that consists of the all Narayana numbers from $N_{1}$ in order. So,
\begin{equation}
\left[ 
\begin{array}{cccc}
N_{1} & N_{2} & ... & N_{a} \\ 
N_{a+1} & N_{a+2} & ... & N_{2a} \\ 
N_{2a+1} & N_{2a+2} & ... & N_{3a} \\ 
... & ... & ... & ...
\end{array}
\right] .
\end{equation}

We shall investigate a third order linear recurrence $N_{m}=p_{a}N_{m-a}+q_{a}N_{m-2a}+N_{m-3a}$ for Narayana numbers with some $p_{a},q_{a}\in \Z$. 

\begin{lemma}\label{lem1}
$N_{m}=N_{m-2}+2N_{m-4}+N_{m-6}$, $N_{m}=4N_{m-3}-3N_{m-6}+N_{m-9}$ and $N_{m}=5N_{m-4}-2N_{m-8}+N_{m-12}$ for any $m\in \Z$.
\end{lemma}
\begin{proof}
Observe that $N_{6}=4=2+2\cdot 1=N_{4}+2N_{2}+N_{0}$. If we assume $N_{t}=N_{t-2}+2N_{t-4}+N_{t-6}$ for all $t<m$. Then,
\begin{align*}
N_{m}&=N_{m-1}+N_{m-3}\\
&=(N_{m-3}+2N_{m-5}+N_{m-7})+(N_{m-5}+2N_{m-7}+N_{m-9})\\
&=(N_{m-3}+N_{m-5})+2(N_{m-5}+N_{m-7})+(N_{m-7}+N_{m-9})\\
&= N_{m-2}+2N_{m-4}+N_{m-6}.
\end{align*}
Similar to this, we notice $N_{9}=13=4\cdot 4-3\cdot 1=4N_{6}-3N_{3}+N_{0}$. If we assume $N_{t}=4N_{t-3}-3N_{t-6}+N_{t-9}$ for all $t<m$, then the induction hypothesis proves $N_{m}=4N_{m-3}-3N_{m-6}+N_{m-9}$. 

Analogously, since $N_{12}=41=5\cdot 9-2\cdot2=5N_{8}-2N_{4}+N_{0}$, the identity $N_{m}=5N_{m-4}-2N_{m-8}+N_{m-12}$ can be proved immediately by induction.
\end{proof}

\begin{remark}
Note that the identity $N_{4r}=5N_{4(r-1)}-2N_{4(r-2)}+N_{4(r-3)}$ is a special case of $N_{m}=5N_{m-4}-2N_{m-8}+N_{m-12}$ in above Lemma \ref{lem1} when $m$ is divisible by 4 ($m=4r$), with $r\in \Z$. We extend Lemma \ref{lem1} to any integer $1\leq a\leq 8$.
\end{remark}

\begin{theorem}\label{teo1} 
Let $1\leq a\leq 8$. Then the third order recurrence $N_{m}=p_{a}N_{m-a}+q_{a}N_{m-2a}+N_{m-3a}$ of $N_{m}$\ holds with the following $(p_{a},q_{a})$.
$$
\begin{tabular}{|l|l|l|}
\hline
$a$ & $(p_{a},q_{a})$ & $
\begin{array}{cccc}
N_{m}=p_{a}N_{m-a}+q_{a}N_{m-2a}+N_{m-3a}
\end{array}
$ \\ \hline
$1$ & $(1,0)$ & $
\begin{array}{cccc}
N_{m}=N_{m-1}+N_{m-3}
\end{array}
$ \\ 
$2$ & $(1,2)$ & $
\begin{array}{cccc}
N_{m}=N_{m-2}+2N_{m-4}+N_{m-6}
\end{array}%
$ \\ 
$3$ & $(4,-3)$ & $
\begin{array}{cccc}
N_{m}=4N_{m-3}-3N_{m-6}+N_{m-9}
\end{array}
$ \\ 
$4$ & $(5,-2)$ & $
\begin{array}{cccc}
N_{m}=5N_{m-4}-2N_{m-8}+N_{m-12}
\end{array}
$ \\ 
$5$ & $(6,5)$ & $
\begin{array}{cccc}
N_{m}=6N_{m-5}+5N_{m-10}+N_{m-15}
\end{array}
$ \\ 
$6$ & $(10,-1)$ & $
\begin{array}{cccc}
N_{m}=10N_{m-6}-N_{m-12}+N_{m-18}
\end{array}
$ \\ 
$7$ & $(15,-7)$ & $
\begin{array}{cccc}
N_{m}=15N_{m-7}-7N_{m-14}+N_{m-21}
\end{array}
$ \\ 
$8$ & $(21,6)$ & $
\begin{array}{cccc}
N_{m}=21N_{m-8}+6N_{m-16}+N_{m-24}
\end{array}
$ \\ \hline
\end{tabular}
$$
\end{theorem}
\begin{proof}
Clearly $N_{m}=N_{m-1}+N_{m-3}$ shows $(p _{1},q _{1})=(1,0)$. And Lemma \ref{lem1} shows $(p_{a},q_{a})=(1,2)$, $(4,-3)$ and $(5,-2)$ for $a=2,3,4$, respectively. 

Let $m=ar+b$ ($1\leq b <a$) and $5\leq a\leq 8$. In order to express $N_{ar+b}$ as $p_{a}N_{a(r-1)+b}+q_{a}N_{a(r-2)+b}+N_{a(r-3)+b}$, we shall consider $a$ columns Narayana tables. Let us begin with $a=5$.
$$
\left[ 
\begin{array}{ccccc}
1 & 1 & 1 & 2 & 3 \\ 
4 & 6 & 9 & 13 & 19 \\ 
28 & 41 & 60 & 88 & 129 \\ 
189 & 277 & 406 & 595 & ...%
\end{array}%
\right] .
$$
Then it can be observed that, for instance
$$
\left\{ 
\begin{array}{c}
N_{16}=189=6\cdot 28+5\cdot 4+1=6N_{11}+5N_{6}+N_{1} \\ 
N_{17}=277=6\cdot 41+5\cdot 6+1=6N_{12}+5N_{7}+N_{2} \\ 
N_{18}=406=6\cdot 60+5\cdot 9+2=6N_{13}+5N_{8}+N_{3}
\end{array}
\right. .
$$
Thus, by assuming $N_{t}=6N_{t-5}+5N_{t-10}+N_{t-15}$ for all $t<m$, the induction hypothesis gives rise to 
\begin{align*}
N_{m}&=N_{m-1}+N_{m-3}\\
&=(6N_{m-6}+5N_{m-11}+N_{m-16})+(6N_{m-8}+5N_{m-13}+N_{m-18})\\
&=6(N_{m-6}+N_{m-8})+5(N_{m-11}+N_{m-13})+(N_{m-16}+N_{m-18}) \\
&=6N_{m-5}+5N_{m-10}+N_{m-15},
\end{align*}
so $(p_{5},q_{5})=(6,5)$. Moreover from the $6$ columns Narayana table
$$
\left[ 
\begin{array}{cccccc}
1 & 1 & 1 & 2 & 3 & 4 \\ 
6 & 9 & 13 & 19 & 28 & 41 \\ 
60 & 88 & 129 & 189 & 277 & ... \\ 
\end{array}%
\right]
$$
we can observe that, for instance 
$$
\left\{ 
\begin{array}{c}
N_{19}=595=10\cdot 60-6+1=10N_{13}-N_{7}+N_{1} \\ 
N_{20}=872=10\cdot 88-9+1=10N_{14}-N_{8}+N_{2} \\ 
N_{21}=1278=10\cdot 129-13+1=10N_{15}-N_{9}+N_{3}
\end{array}%
\right. .
$$
By assuming $N_{t}=10N_{t-6}-N_{t-12}+N_{t-18}$ for all $t<m$, we have 
\begin{align*}
N_{m}&=N_{m-1}+N_{m-3}\\
&=(10N_{m-7}-N_{m-13}+N_{m-19})+(10N_{m-9}-N_{m-15}+N_{m-21})\\
&=10(N_{m-7}+N_{m-9})-(N_{m-13}+N_{m-15})+(N_{m-19}+N_{m-21}) \\
&=10N_{m-6}+5N_{m-12}+N_{m-18},
\end{align*}
so $(p_{6},q_{6})=(10,-1)$. Therefore the observations and mathematical induction prove that the coefficients $(p_{a},q_{a})$ for $a=7,8$ satisfying $$N_{m}=p_{a}N_{m-a}+q_{a}N_{m-2a}+N_{m-3a}$$
are equal to $(15,-7)$ and $(21,6)$, respectively.
\end{proof}

\begin{remark}
We note that the subscript $m$ of $N_{m}$ could be negative, for example, in $6\ $columns Narayana table, $N_{15}=129=10N_{9}-N_{3}+N_{-3}$.
\end{remark}

\begin{definition}
Let $r\in \mathbb{Z}$. A sequence $n_{r}$ is called a Narayana type if it satisfies $n_{r}+n_{r+2}-n_{r+3}=0$ with any initials $n_{1},$ $n_{2}$ and $n_{3}$.
\end{definition}

\begin{theorem}\label{th:3}
For $1\leq a\leq 8$, let $(p_{a},q _{a})$ be the coefficient of the third order recurrence $N_{m}=p_{a}N_{m-a}+q_{a}N_{m-2a}+N_{m-3a}$.
Then,
\begin{enumerate}
\item For $1\leq s\leq 5$, $\{p_{s}\}$ is a Narayana type sequence $p_{s+3}=p_{s+2}+p_{s}$ with initials $p_{1}=p_{2}=1$ and $p_{3}=4$, while $\{q_{s}\}$\ satisfies $q_{s+3}=q_{s}-q_{s+1}$ with $q_{1}=0$, $q_{2}=2$ and $q_{3}=-3$.
\item Moreover, $p _{s}=N_{s}+3N_{s-2}$ and $q_{s}=-p_{-s}$ for $1\leq s\leq 8$, where $N_{s}$ is the $s$-th Narayana number.
\end{enumerate}
\end{theorem}
\begin{proof}
By above Theorem \ref{teo1}, $$\{p_{s}\}_{s=1}^{8}=\{1,1,4,5,6,10,15,21\}$$ and $$\{q_{s}\}_{s=1}^{8}=\{0,2,-3,-2,5,-1,-7,6\}.$$ So it is easy to see that $p_{s+3}=p_{s+2}+p_{s}$ and $q_{s+3}=q_{s}-q_{s+1}$ for $1\leq s\leq 5$. Moreover, by means of Narayana numbers $N_{m}$, we notice
$$
p_{1}=1=N_{1}+3N_{-1},\ p_{2}=1=N_{2}+3N_{0},\ p_{3}=4=N_{3}+3N_{1},
$$
and $p_{4}=p_{3}+p_{1}=5=N_{4}+3N_{2}$, etc. So $p_{s}=N_{s}+3N_{s-2}$ for $1\leq s\leq 8.$ Now, by considering $N_{m}$ with negative $m$, the Narayana type sequence $\{p_{s}\}$ can be extended to
any $s\in \mathbb{Z}$, as follows.
$$
\begin{tabular}{l|llllllllllllll}
\hline
$s$ & $...$ & $-8$ & $-7$ & $-6$ & $-5$ & $-4$ & $-3$ & $-2$ & $-1$ & $0$ & $%
1$ & $2$ & $3$ & $...$ \\ \hline
$p_{s}$ & $...$ & $-6$ & $7$ & $1$ & $-5$ & $2$ & $3$ & $-2$ & $0$ & $3$
& $1$ & $1$ & $4$ & $...$ \\ \hline
\end{tabular}%
$$
Then by comparing $\{p_{s}\}_{s=-1}^{-8}=\{0,-2,3,2,-5,1,7,-6\}$ with $\{q_{s}\}_{s=1}^{8}$, we find that $q_{s}=-p_{-s}$ for $0\leq s\leq 8$.
\end{proof}

\section{The third order linear recurrence of $N_{m}$}
We shall generalize the findings in above section for $0\leq a\leq 8$ to any integer $a$.
\begin{theorem}\label{teo2}
Let $p_{a}=N_{a}+3N_{a-2}$ and $q_{a}=-p_{-a}$ for any $a\in \mathbb{Z}^{+}$. Then, any $m$-th Narayana number satisfies $N_{m}=p_{a}N_{m-a}+q_{a}N_{m-2a}+N_{m-3a}$ for every $a<m$.
\end{theorem}
\begin{proof}
It is due to above Theorem \ref{th:3}  if $0\leq a\leq 8$. Since $p_{a}=N_{a}+3N_{a-2}$ for all $a\in \mathbb{Z}^{+}$, then $\{p_{a}\}$ is a Narayana type sequence because 
\begin{align*}
p_{a}+p_{a+2} &=(N_{a}+3N_{a-2})+(N_{a+2}+3N_{a}) \\
&=(N_{a}+N_{a+2})+3(N_{a-2}+N_{a}) \\
&=N_{a+3}+3N_{a+1}=p_{a+3}.
\end{align*}

Similarly, since $q_{a}=-p_{-a}$ for all $a$, $\{q_{a}\}$ satisfies 
\begin{align*}
q_{a}-q _{a+1} &=-p_{-a}+p _{-(a+1)} \\
&=-(p_{-a-1}+p_{-a-3})+p_{-(a+1)} \\
&=-p_{-(a+3)}=q_{a+3}.
\end{align*}

We now suppose that the three order recurrence $N_{m}=p_{t}N_{m-t}+q_{t}N_{m-2t}+N_{m-3t}$ hold for all $t\leq a$. Since 
\begin{align*}
N_{m-(a-2)}&=N_{m-(a-1)}+N_{m-(a+1)}\\
N_{m-2(a-2)}&=N_{m-2(a-1)}+2N_{m-2a}+N_{m-2(a+1)}\\
N_{m-3(a-2)}&=4N_{m-3(a-1)}-3N_{m-3a}+N_{m-3(a+1)}.
\end{align*}
Then, by lemma \ref{lem1} and the mathematical induction with long calculations proves that 
\begin{align*}
p_{a+1}N_{m-(a+1)}&+q_{a+1}N_{m-2(a+1)}+N_{m-3(a+1)} \\
&=(p_{a}+p_{a-2})N_{m-(a+1)}+(q_{a-2}-q_{a-1})N_{m-2(a+1)}+N_{m-3(a+1)} \\
&=(p_{a}+p_{a-2})(N_{m-(a-2)}-N_{m-(a-1)}) \\
& \ \ +(q_{a-2}-q_{a-1})(N_{m-2(a-2)}-N_{m-2(a-1)}-2N_{m-2a}) \\
& \ \ +(N_{m-3(a-2)}-4N_{m-3(a-1)}+3N_{m-3a}) \\
&= N_{m}.
\end{align*}
\end{proof}

Theorem \ref{teo2} provides a good way to find huge Narayana numbers. For instance, for $40$-th Narayana number $N_{40}$, we may choose any $a$, say $a=10$, then $p_{10}=N_{10}+3N_{8}=46$ and $q_{10}=-p _{-10}=-13$, thus 
\begin{align*}
N_{40}&=p_{10}N_{40-10}+q_{10}N_{40-20}+N_{40-30}\\
&=46\cdot 39865-13\cdot 872+19 \\
&=1822473,
\end{align*}
a $7$ digit integer. On the other hand, if we take $a=8$ then $p_{8}=N_{8}+3N_{6}=21$ and $q_{8}=-p_{-8}=6$, so $P_{40}$ can be obtained by $N_{40}=p_{8}N_{32}+q_{10}N_{24}+N_{16}$. 

More identities for $p _{a}$ can be developed in terms of three successive Narayana numbers.

Now, for each $m \in \Z^{+}$, we define two sequences
\begin{equation}
P_{N,m}=N_{m}+N_{-m}\ \textrm{and}\ Q_{N,m}=N_{m}-N_{-m}.
\end{equation}
Then it is easy to have the table
$$
\begin{tabular}{l|llllllllllllll}
\hline
$m$ & $1$ & $2$ & $3$ & $4$ & $5$ & $6$ & $7$ & $8$ & $9$ & $10$ & $%
11$ & $12$ & $...$ \\ \hline\hline
$N_{m}$  & $1$ & $1$ & $1$ & $2$ & $3$ & $4$ & $6$ & $9$ & $%
13 $ & $19$ & $28$ & $41$ & $...$ \\ 
$N_{-m}$  & $0$ & $1$ & $0$ & $-1$ & $1$ & $1$ & $-2$ & $0$ & $3$ & $%
-2 $ & $-3$ & $5$ & $...$ \\ \hline
$P_{N,m}$ & $1$ & $2$ & $1$ & $1$ & $4$ & $5
$ & $4$ & $9$ & $16$ & $17$ & $25$ & $46$ & $...$ \\ 
$Q_{N,m}$ & $1$ & $0$ & $1$ & $3$ & $2$ & $3
$ & $8$ & $9$ & $10$ & $21$ & $31$ & $36$ & $...$ \\ \hline
\end{tabular}
$$
From the table, we notice $N_{8}=9=46-21-16$ or $N_{8}=P_{N,12}-Q_{N,10}-P_{N,9}$.

\begin{theorem}
Let $m\in\mathbb{Z}^{+}$. Then the sequences $\{N_{m}\}$ satisfy the relation 
$$N_{m}=P_{N,m+4}-Q_{N,m+2}-P_{N,m+1}.$$ Furthermore, $N_{m}=\frac{1}{2}\left(
P_{N,m}+Q_{N,m}\right) $ and $N_{-m}=\frac{1}{2}\left( P_{N,m}-Q_{N,m}\right)$.
\end{theorem}
\begin{proof}
It is easy to see that 
\begin{eqnarray*}
P_{N,m} &=&N_{m}+N_{-m} \\
&=&\left( N_{m-1}+N_{m-3}\right) +\left( -N_{-(m-2)}+N_{-(m-3)}\right) \\
&=&P_{N,m-3}+N_{m-1}-N_{-(m-2)} \\
&=&P_{N,m-3}+(N_{m-2}+N_{m-4})-N_{-(m-2)} \\
&=&P_{N,m-3}+Q_{N,m-2}+N_{m-4}.
\end{eqnarray*}
Hence $N_{m}=P_{N,m+4}-Q_{N,m+2}-P_{N,m+1}$.
\end{proof}

\begin{theorem}\label{teo4}
Let $m=ar+b$ with $1\leq b\leq a<m$. Let $(p_{a},q_{a})$ be the coefficient of the third order recurrence $N_{m}=p_{a}N_{m-a}+q_{a}N_{m-2a}+N_{m-3a}$. Then, $N_{m}$ is a linear combination of any three consecutive entries of $b$-th column in the $a$ columns Narayana table. Furthermore, $N_{m}$ is expressed by the first three terms $N_{2a+b}$, $N_{a+b}$ and $N_{b}$ of $b$-th column.
\end{theorem}
\begin{proof}
Let $N_{m}=p_{a}N_{m-a}+q_{a}N_{m-2a}+N_{m-3a}$ in Theorem \ref{teo2}. Then,
\begin{eqnarray*}
N_{ar+b} &=&p_{a}N_{a(r-1)+b}+q_{a}N_{a(r-2)+b}+N_{a(r-3)+b} \\
&=&p_{a}\left( p_{a}N_{a(r-2)+b}+q_{a}N_{a(r-3)+b}+N_{a(r-4)+b}\right) \\
&&\ \ +q{a}N_{a(r-2)+b}+N_{a(r-3)+b} \\
&=&\left( p_{a}^{2}+q_{a}\right) N_{a(r-2)+b}+\left(p_{a}q_{a}+1\right) N_{a(r-3)+b}+p_{a}N_{a(r-4)+b}.
\end{eqnarray*}
Hence after $s$ step (with $0<s<r$), if we write
$$N_{ar+b}=\alpha N_{a(s+2)+b}+\beta N_{a(s+1)+b}+\gamma N_{as+b},$$ with $\alpha, \beta,\gamma \in \mathbb{Z}$. Then, in the next step we have $N_{ar+b}=(p_{a}\alpha +\beta)N_{a(s+1)+b}+(q_{a}\alpha+\gamma) N_{as+b}+ \alpha N_{as+b}$. Continue this process to reach $s=1$, then it follows that $N_{m}$ is a linear combination of $N_{2a+b}$, $N_{a+b}$ and $N_{b}$.
\end{proof}

For example, for $N_{38}$ we may take any $a<38$, say $a=7$. Since $(p_{7},q_{7})=(15,-7)$ by Theorem \ref{teo1}, $N_{38}$ can be obtained easily by Theorem \ref{teo4} that 
\begin{eqnarray*}
N_{38} &=&15N_{31}-7N_{24}+N_{17} \\
&=&(15^{2}-7)N_{24}+(15\cdot (-7)+1)N_{17}+15N_{10} \\
&=&218N_{24}-104N_{17}+15P_{10} \\
&=&218\left( 15N_{17}-7N_{10}+N_{3}\right) -104N_{17}+15N_{10} \\
&=&3166N_{17}-1511N_{10}+218N_{3} \\
&=&3166\cdot 277-1511\cdot 19+218\cdot 1 \\
&=&848491.
\end{eqnarray*}

However, since $N_{m}$ is composed of $N_{m-a}$, $N_{m-2a}$ and $N_{m-3a}$, it may be better to choose $a\approx \frac{m}{3}$. Indeed if we take $\frac{38}{3}\approx 12=a$, then $N_{38}=p_{12}N_{26}-p_{-12}N_{14}+N_{2}$ and the last term $N_{2}=1$ is known easily.

\begin{remark}
Assume the same context $(p_{a},q_{a})$ as before. If $m=3a$, then $N_{m}=p_{\frac{m}{3}}N_{\frac{2m}{3}}+q_{\frac{m}{3}}N_{\frac{m}{3}}$ since $N_{0}=0$. In the other hand, if $m=3a+1$ or $m=3a+2$ and $N_{1}=N_{2}=1$, it follows that 
\begin{equation*}
N_{m}=p_{\left\lfloor \frac{m}{3}\right\rfloor }N_{\left\lfloor \frac{2m}{3}\right\rfloor +1}+q_{\left\lfloor \frac{m}{3}\right\rfloor
}N_{\left\lfloor \frac{m}{3}+\frac{1}{2}\right\rfloor +1}+1,
\end{equation*}
where $\left\lfloor x \right\rfloor$ is the floor function of $x$.
\end{remark}
For example, if $m=26,$ we have $N_{26}=p_{8}N_{18}+q_{8}N_{10}+1=8641=21\cdot 406+6\cdot 19+1$. If $m=36,$ we have $N_{36}=p_{12}N_{24}+q_{12}N_{12}$.

\section{Partial sum of Narayana numbers in a row}

Consider $N_{ar+b}$ ($r\geq 0$ and $1\leq b\leq a$) as an entry placed at the $(r+1)$-th row and $b$-th column in the table, and let
\begin{equation*}
S_{N,r}^{(a,b)}=\sum_{k=0}^{r}N_{ak+b}=N_{b}+N_{a+b}+N_{2a+b}+\cdots +N_{ar+b}
\end{equation*}
be the partial sum of $r+1$ entries of $b$-th column.

\begin{theorem}\label{teo5}
For $r\geq 0$, we have 
\begin{equation}\label{s2}
S_{N,r}^{(4,0)}=\sum_{k=0}^{r}N_{4k}=\frac{1}{3}\left(N_{4(r+1)}-N_{4r}+N_{4(r-1)}-1\right).
\end{equation}
\end{theorem}
\begin{proof}
Let $r=3$, Lemma \ref{lem1} shows $N_{4(4)}=5N_{4(3)}-2N_{4(2)}+N_{4}$, so we have 
\begin{eqnarray*}
3\sum_{k=0}^{3}N_{4k}&=&3N_{4(3)}+3N_{4(2)}+3N_{4}+3N_{0} \\
&=&(5N_{4(3)}-2N_{4(2)}+N_{4})-2N_{4(3)}+5N_{4(2)}+2N_{4}+3N_{0}  \\
&=&N_{4(4)}-2N_{4(3)}+5N_{4(2)}+2N_{4}\\
&=&N_{4(4)}-2N_{4(3)}+(N_{4(3)}+2N_{4})+2N_{4}\\
&=&N_{4(4)}-N_{4(3)}+N_{4(2)}-1,
\end{eqnarray*}
since $4N_{4}=N_{8}-1$. Assume $3\sum_{k=0}^{r}N_{4k}=N_{4(r+1)}-N_{4r}+N_{4(r-1)}-1$ is true. Then it follows that 
\begin{align*}
N_{4(r+2)}-&N_{4(r+1)}+N_{4r}-1\\
&=\left(5N_{4(r+1)}-2N_{4r}+N_{4(r-1)}\right) -N_{4(r+1)}+N_{4r}-1 \\
&=4N_{4(r+1)}-N_{4r}+N_{4(r-1)}-1\\
&=3N_{4(r+1)}+N_{4(r+1)}-N_{4r}+N_{4(r-1)}-1\\
&=3N_{4(r+1)}+3\sum_{k=0}^{r}N_{4k}=3\sum_{k=0}^{r+1}N_{4k}.
\end{align*}
\end{proof}

\begin{remark}
Theorem \ref{teo5} is a sum of $4k$ subscripted Narayana numbers. But in our context, Eq. (\ref{s2}) can be explained as a sum of entries of $4$-th column in the $4$ columns Narayana table. We now shall study the sum of entries of any $b$-th column in the $4$ columns Narayana table.
\end{remark}

\begin{theorem}
Consider $S_{N,r}^{(4,b)}$\ with $1\leq b\leq 4$. Then for $r\geq 3$,%
\begin{equation}\label{s5}
S_{N,r}^{(4,b)}=\left\{ 
\begin{array}{ccc}
5S_{N,r-1}^{(4,b)}-2S_{N,r-2}^{(4,b)}+S_{N,r-3}^{(4,b)}-1 & if & b=1, \\ 
5S_{N,r-1}^{(4,b)}-2S_{N,r-2}^{(4,b)}+S_{N,r-3}^{(4,b)}+1 & if & b=2,4, \\ 
5S_{N,r-1}^{(4,b)}-2S_{N,r-2}^{(4,b)}+S_{N,r-3}^{(4,b)}+2& if & b=3. 
\end{array}
\right.
\end{equation}
\end{theorem}
\begin{proof}
The $4$ columns Narayana table makes the table of $S_{N,r}^{(4,b)}$ as follows. 
\begin{equation*}
\left[ 
\begin{tabular}{llll}
$1$ & $1$ & $1$ & $2$ \\ 
$3$ & $4$ & $6$ & $9$ \\ 
$13$ & $19$ & $28$ & $41$ \\ 
$60$ & $88$ & $129$ & $189$ \\ 
$277$ & $406$ & $595$ & $\cdots $%
\end{tabular}%
\right] \text{ and 
\begin{tabular}{l|llll}
\hline
$r$ & $S_{N,r}^{(4,1)}$ & $S_{N,r}^{(4,2)}$ & $S_{N,r}^{(4,3)}$ & $S_{N,r}^{(4,4)}$
\\ \hline
$0$ & $1$ & $1$ & $1$ & $2$ \\ 
$1$ & $4$ & $5$ & $7$ & $11$ \\ 
$2$ & $17$ & $24$ & $35$ & $52$ \\ 
$3$ & $77$ & $112$ & $164$ & $241$ \\ 
$4$ & $354$ & $518$ & $759$ & $\cdots $ \\ \hline
\end{tabular}%
.}
\end{equation*}%
When $r=4$ and $b=1$, we notice $77=\left( 5\cdot 17-2\cdot 4+1\right)-1$, and it
can be written as 
\begin{equation*}
S_{N,3}^{(4,1)}=5S_{N,2}^{(4,1)}-2S_{N,1}^{(4,1)}+S_{N,0}^{(4,1)}-1.
\end{equation*}%
Similar to this, we observe that 
\begin{equation*}
\left\{ 
\begin{array}{c}
S_{N,3}^{(4,2)}=5S_{N,2}^{(4,2)}-2S_{N,1}^{(4,2)}+S_{N,0}^{(4,2)}+1,\\ 
S_{N,3}^{(4,3)}=5S_{N,2}^{(4,3)}-2S_{N,1}^{(4,3)}+S_{N,0}^{(4,3)}+2,\\ 
S_{N,3}^{(4,4)}=5S_{N,2}^{(4,4)}-2S_{N,1}^{(4,4)}+S_{N,0}^{(4,4)}+1.%
\end{array}%
\right. .
\end{equation*}%
Furthermore, assume that $S_{N,r}^{(4,b)}=5S_{N,r-1}^{(4,b)}-2S_{N,r-2}^{(4,b)}+S_{N,r-3}^{(4,b)}+1$ if $b=2,4$. Then Theorem \ref{teo1} together with induction hypothesis yields
\begin{align*}
S_{N,r+1}^{(4,b)}&=\sum_{k=0}^{r+1}N_{4k+b}=S_{N,r}^{(4,b)}+N_{4(r+1)+b} \\
&=(5S_{N,r-1}^{(4,b)}-2S_{N,r-2}^{(4,b)}+S_{N,r-3}^{(4,b)}+1)+(5N_{4r+b}-2N_{4(r-1)+b}+N_{4(r-2)+b}) \\
&=5S_{N,r}^{(4,b)}-2S_{N,r-1}^{(4,b)}+S_{N,r-2}^{(4,b)}+1,
\end{align*}
this proves the Eq. (\ref{s5}). Similarly, other relationships are followed.
\end{proof}

\end{document}